\documentclass[12pt]{amsart}

\usepackage{graphicx}

\usepackage{graphicx}
\usepackage{amssymb}
\usepackage{amsfonts}
\usepackage{latexsym}
\usepackage{amscd}
\usepackage{ae,enumerate}
\usepackage[mathscr]{euscript}
\usepackage{xy} \xyoption{all}

\newtheorem{thm}{Theorem}[section]
\newtheorem{lem}[thm]{Lemma}
\newtheorem{rem}[thm]{Remark}

\newtheorem{cor}[thm]{Corollary}
\newtheorem{prop}[thm]{Proposition}

\newtheorem{defn}[thm]{Definition}
\newtheorem{definition}[thm]{Definition}

\theoremstyle{remark}

\numberwithin{equation}{section}

\begin{document}

\title[Prime non-primitive von Neumann regular algebras]{On prime non-primitive von Neumann regular algebras
}




\author{Gene Abrams}
\address{Department of Mathematics, University of Colorado,
Colorado Springs CO 80933 U.S.A.} \email{abrams@math.uccs.edu}
\author{Jason P. Bell}
\address{Department of Mathematics, Simon Fraser University, Burnaby, BC, Canada V5A1S6} \email{jpb@math.sfu.ca}
\author{Kulumani M. Rangaswamy}
\address{Department of Mathematics, University of Colorado,
Colorado Springs CO 80933 U.S.A.}\email{krangasw@uccs.edu}
\thanks{The first author is partially supported by the U.S. National Security Agency under grant number H89230-09-1-0066. The second author was supported by NSERC grant 31-611456.}




\begin{abstract}

Let $E$ be any directed graph, and $K$ any field.  We classify those  graphs $E$ for which the Leavitt path algebra $L_K(E)$ is  primitive.  As a consequence, we obtain classes of examples of von Neumann regular prime rings which are not primitive.
\end{abstract}

\maketitle

\bibliographystyle{plain}

\section{Introduction, and the Primeness Theorem}

\bigskip

The structure of prime and primitive algebras has long been the focus of much attention.   The spark for much of the interest in such structures was a question  posed by Kaplansky \cite[p. 2]{K}:  ``{\it Is a regular prime ring necessarily primitive?}"    This  question was, in its time,  ``{\it  ... the major question in the theory of von Neumann regular rings}" (\cite[p. 308]{R}).   An example of such a ring, a specific type of group algebra, was given first in 1977 by Domanov (see \cite{Dom}; see also \cite{P} for a more complete account).

 During the past few years, the {\it Leavitt path algebras} have  been actively investigated, thereby leading to ever-increasing understanding of various algebraic properties thereof. In particular, both the prime and the primitive Leavitt path algebras $L_K(E)$ associated with row-finite graphs $E$ and arbitrary fields $K$ have been identified in \cite{APPSM2}.
In the current article we provide necessary and sufficient conditions on an arbitrary directed graph $E$ which ensure that the Leavitt path algebra $L_K(E)$ is prime (Theorem \ref{primenesstheorem}), and which ensure that $L_K(E)$ is primitive (Theorem \ref{primitivitytheorem}).   As we shall see, the primeness condition in the general case matches exactly the primeness condition given in \cite{APPSM2} for row-finite graphs.  On the other hand, the primitivity of $L_K(E)$ in the general case requires an additional condition (the ``Countable Separation Property") on the graph $E$, a condition which is automatically satisfied for row-finite graphs $E$ having prime $L_K(E)$.

If $u$ and $v$ are vertices in the graph $E$, we write $u\geq v$ in case $u = v$ or there is a directed path in $E$ whose source vertex is $u$ and range vertex is $v$.  If $c = e_1e_2\cdots e_n$ is a cycle in $E$, then an {\it exit} for $c$ is an edge $f$ which is not equal to any of the $e_i$, but whose source vertex is equal to the source vertex of one of the $e_i$.   With this notation in hand, we now present our main result (Theorem \ref{primitivitytheorem}), in which we establish an easy-to-verify list of properties of $E$ which taken together are equivalent to the primitivity of $L_K(E)$.

\medskip

{\bf Theorem.} Let $E$ be an arbitrary graph, and $K$ any field.  Then $L_K(E)$ is primitive if and only if
\begin{enumerate}[(i)]
\item  For every pair of vertices $u,v$ in $E$ there exists a vertex $w$ in $E$ for which $u\geq w$ and $v\geq w$,
\item  Every cycle in $E$ has an exit, and
\item  There exists a countable set $S$ of vertices in $E$ for which, for each vertex $u$ in $E$, there exists $w\in S$ for which $u\geq w$.
\end{enumerate}
\noindent
These conditions on $E$ are called, respectively, {\it Condition (MT3)}, {\it Condition (L)}, and the {\it Countable Separation Property}.

\medskip

Conditions on the graph $E$ which imply the von Neumann regularity of $L_K(E)$ were established in \cite{AR}.   With such information in hand, Theorems  \ref{primenesstheorem} and \ref{primitivitytheorem}
 then immediately allow us to produce a bumper crop of examples of prime non-primitive von Neumann regular algebras.  Four such infinite classes of algebras (two of which consist of unital algebras) are presented in Section \ref{primenotprimitivevNrsection}.

We begin by giving a condensed reminder of the germane definitions.  For a more complete description and discussion, see for example \cite{R} and \cite{AA3}.   A ring $R$ is called {\it prime} if the product of any two nonzero two-sided ideals of $R$ is again nonzero;  {\it left primitive} if $R$ admits a simple faithful left $R$-module; and {\it von Neumann regular} (or sometimes more concisely {\it regular}) in case for each $a\in R$ there exists $x\in R$ for which $a=axa$.  It is easy to show that any primitive ring is prime.  Furthermore, in a prime ring, the intersection of any two nonzero two-sided ideals is nonzero, as the intersection contains the product.

   A \emph{(directed) graph} $E=(E^0,E^1,r,s)$ consists of two  sets $E^0,E^1$ and functions
$r,s:E^1 \to E^0$.    The elements of $E^0$ are called \emph{vertices} and the elements of $E^1$ \emph{edges}. The sets $E^0$ and $E^1$ are allowed to be of arbitrary cardinality.  A \emph{path} $\mu$ in a graph $E$ is a finite sequence of edges
$\mu=e_1\dots e_n$ such that $r(e_i)=s(e_{i+1})$ for $i=1,\dots,n-1$. In this case, $s(\mu):=s(e_1)$ is the
\emph{source} of $\mu$, $r(\mu):=r(e_n)$ is the \emph{range} of $\mu$, and $n := \ell(\mu)$ is the \emph{length} of $\mu$.
 We view the elements of $E^0$ as paths of length $0$.  If $\mu = e_1...e_n$ is a path in $E$, and if $v=s(\mu)=r(\mu)$ and $s(e_i)\neq s(e_j)$ for every $i\neq j$, then $\mu$ is called a
\emph{cycle based at} $v$.     A graph which contains no cycles is called \emph{acyclic}. An edge $e$ is an {\it exit} for a path $\mu = e_1 \dots e_n$ if there exists $i$ such
that $s(e)=s(e_i)$  and $e\neq e_i$.
If
$s^{-1}(v)$ is a finite set for every $v\in E^0$, then the graph $E$ is called \emph{row-finite}.

\medskip

\begin{definition}\label{definition}  {\rm Let $E$ be any directed graph, and $K$ any field. The
{\em Leavitt path $K$-algebra} $L_K(E)$ {\em of $E$ with coefficients in $K$} is  the $K$-algebra generated by a set
$\{v\mid v\in E^0\}$, together with a set of variables $\{e,e^*\mid e\in E^1\}$,
which satisfy the following relations:

(V)  $vw = \delta_{v,w}v$ for all $v,w\in E^0$ \  (i.e., $\{v\mid v\in E^0\}$ is a set of orthogonal idempotents).

  (E1) $s(e)e=er(e)=e$ for all $e\in E^1$.

(E2) $r(e)e^*=e^*s(e)=e^*$ for all $e\in E^1$.

 (CK1) $e^*e'=\delta _{e,e'}r(e)$ for all $e,e'\in E^1$.

(CK2)  $v=\sum _{\{ e\in E^1\mid s(e)=v \}}ee^*$ for every vertex $v\in E^0$ having $1\leq |s^{-1}(v)| < \infty$.
}

\end{definition}

\medskip


 We let $r(e^*)$ denote $s(e)$, and we let $s(e^*)$ denote $r(e)$.
If $\mu = e_1 \dots e_n$ is a path, then we denote by $\mu^*$ the element $e_n^* \dots e_1^*$ of $L_K(E)$.


Many well-known algebras arise as the Leavitt path algebra of a
 graph.
For example, the classical Leavitt $K$-algebra $L_K(1,n)$ for $n\ge 2$; the full $n\times n$ matrix algebra ${\rm M}_n(K)$ over $K$; and the Laurent polynomial
algebra $K[x,x^{-1}]$ arise, respectively, as the Leavitt path algebras of the ``rose with $n$ petals" graph $R_n$ ($n\geq 2$); the oriented line graph $A_n$ having $n$ vertices; and the ``one vertex, one loop" graph $R_1$ pictured here.

$$R_n \ =  
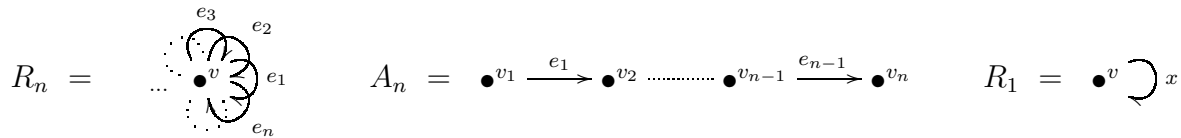
\xymatrix{ & {\bullet^v} \ar@(ur,dr) ^{e_1} \ar@(u,r) ^{e_2}
\ar@(ul,ur) ^{e_3} \ar@{.} @(l,u) \ar@{.} @(dr,dl) \ar@(r,d) ^{e_n}
\ar@{}[l] ^{\ldots} } \ \  \ \ \ \  \ A_n \ = \ \xymatrix{{\bullet}^{v_1} \ar [r] ^{e_1} & {\bullet}^{v_2}  \ar@{.}[r] & {\bullet}^{v_{n-1}} \ar [r]
^{e_{n-1}} & {\bullet}^{v_n}} \ \ \ \ \  \ R_1 \ = \   \xymatrix{{\bullet}^{v} \ar@(ur,dr) ^x}$$

\medskip

A (possibly nonunital) ring $R$ is called a {\it ring with local units} in case for each finite subset $S\subseteq R$ there is an idempotent $e\in R$ with $S\subseteq eRe$. If $E$ is a graph for which $E^0$ is finite then we have $\sum _{v\in E^0} v$ is the multiplicative identity in $L_K(E)$; otherwise, $L_K(E)$ is a ring with a set of local units
consisting of sums of distinct vertices. Conversely, if $L_K(E)$ is unital, then $E^0$ is finite.

The Leavitt path algebra $L_K(E)$ is a
${\mathbb Z}$-graded $K$-algebra, spanned as a $K$-vector space by $ \{pq^* \mid p,q$ are paths in $E\}$.  (Recall that the elements of $E^0$ are viewed as paths of length $0$, so that this set includes elements of the form $v$ with $v\in E^0$.)
In particular, for each $n\in \mathbb{Z}$, the degree $n$ component $L_K(E)_n$ is spanned by elements of the form  $\{pq^* \mid {\rm length}(p)-{\rm length}(q)=n\}$.

\smallskip

Throughout, the word ``graph" will mean ``directed graph", and the word ``ideal" will mean ``two-sided ideal" (unless otherwise indicated).   If $u,v$ are vertices in the graph $E$, we write
$u\geq v$ in case there is a path $p$ in $E$ for which $s(p)=u, r(p)=v$.  If $p$ and $q$ are paths in $E$, we say $p$ is an {\it initial subpath} of $q$ in case there is a path $p'$ in $E$ for which $q = pp'$.   If $I$ is an ideal of $L_K(E)$, and $u,v\in E^0$ for which $u\in I$ and $u\geq v$, then $v = p^*p = p^*up \in I$.

We note a basic result about Leavitt path algebras, thereby reiterating an observation which was made in \cite[Introduction]{APRV}.

\begin{prop}\label{L(E)isoL(E)op}
Let $E$ be an arbitrary graph, and $K$ any field.  Then there is a $K$-algebra isomorphism $L_K(E)^{op} \cong L_K(E)$.
\end{prop}

\begin{proof}
The map $\varphi: L_K(E) \to L_K(E)$ defined by
$$\varphi:  \sum_{i=1}^n \lambda_i\alpha_i\beta_i^* \mapsto \sum_{i=1}^n \lambda_i\beta_i\alpha_i^*$$
is a $K$-linear involution on $L_K(E)$ (see e.g. \cite[Remark 3.4]{T}).  The result then follows immediately from general ring-theoretic properties (see e.g. \cite[p. 47]{F}).
\end{proof}

As a consequence of this result, throughout the sequel we drop the left/right designation, and simply refer to  the {\it primitivity} of $L_K(E)$.
Two key properties of a graph $E$ will play a significant role here.
\begin{defn}
{\rm Let $E$ be any graph.
\begin{enumerate}
\item $E$ {\it satisfies Condition (MT3)} in case  for every pair of vertices $v,w$ in $E$, there is a vertex $u$ in $E$ such that $v\geq u$ and $w\geq u$.
\item $E$ {\it satisfies Condition (L)} in case every cycle in $E$ has an exit.
    \end{enumerate}
    }
\end{defn}

 In  \cite[Proposition 5.6]{APPSM1}, necessary and sufficient conditions on the row-finite graph $E$ are presented which ensure that the Leavitt path algebra $L_K(E)$ is prime. (As is the case with many results in this area, the structure of the field $K$ plays no role in these conditions.)    We extend this result to arbitrary graphs $E$, while simultaneously giving a somewhat more streamlined argument.

\begin{thm}\label{primenesstheorem}
Let $E$ be an arbitrary graph, and $K$ any field. Then $L_{K}(E)$ is a prime ring if and only if $E$ satisfies Condition (MT3).
\end{thm}

\begin{proof}
  Suppose $R=L_{K}(E)$ is a prime ring. Let $v,w\in
E^{0}$. Since the ideals $RvR$ and $RwR$ are nonzero, $RvRwR$, and hence $vRw$, is nonzero. This means $v\alpha \beta ^{\ast }w\neq 0$ for some real
paths $\alpha, \beta$, which in particular yields $s(\alpha )=v$ and $s(\beta )=w$. Then $u=r(\alpha )=r(\beta )$ is the desired vertex.

Now suppose that $E$ satisfies Condition (MT3).   Since $L_K(E)$ is graded by the ordered group ${\mathbb Z}$, by  \cite[Proposition II.1.4]{NvO} we see that,  to establish the primeness of $L_K(E)$, we need only show that  $IJ \neq  \{0\}$ for any pair $I,J$ of nonzero {\it graded} ideals of $L_K(E)$.  By \cite[Corollary 3.3(i)]{ABGM}, each nonzero graded ideal of $L_K(E)$ contains a vertex.
 Let $v\in I\cap E^{0}$ and $w\in J\cap E^{0}$. By Condition (MT3) there is a vertex $u$ such that $v\geq u$ and $w\geq u$. But then
$u\in I$ and $u\in J$, so that $0\neq u = u^2 \in IJ$ as desired.
\end{proof}

We note that M. Siles Molina has also independently achieved the result of Theorem \ref{primenesstheorem}.

\section{Primitive Leavitt path algebras:  the row-finite case}\label{rowfinitesection}

In this section we present necessary and sufficient conditions on a row-finite graph $E$ which ensure that the Leavitt path algebra $L_K(E)$ is primitive.

We begin by stating two very useful results; these  will be utilized not only in the current section, but indeed in all the remaining sections of this article.

\begin{lem}\label{LRSLemmas} \cite[Lemmas 2.1 and 2.2]{LRS} \
Let $K$ be any field, and let $R$ be a prime $K$-algebra.  Then there exists a prime unital $K$-algebra $R_1$ into which $R$ embeds as an ideal.  Furthermore,  $R_1$ is primitive if and only if $R$ is primitive.
\end{lem}

\begin{lem}\label{primitivitylemma}  \cite[Theorem 1]{Form}  or \cite[Lemma 11.28]{Lam} \
A unital ring $A$  is left  primitive if and only if there
is a  left  ideal $M \neq A$ of $A$ such that for every nonzero two-sided
ideal $I$ of $A$, $M+I=A$.
\end{lem}

Using these two results, we are already in position to establish the row-finite case.  As it turns out, we will also be able to establish the row-finite result as a specific case of Theorem \ref{primitivitytheorem}. However, the argument in the row-finite case provides clarification and insight to the general case, so we present it separately.

\begin{thm}\label{primitiverowfinite}
Let $E$ be a row-finite graph, and $K$ any field. Then $R=L_{K}(E)$ is  primitive
 if and only if

(i) $E$ satisfies Condition (MT3), and

(ii) $E$ satisfies Condition (L).
\end{thm}

\begin{proof}

First, suppose $E$ satisfies the two conditions.   By Theorem \ref{primenesstheorem}, the (MT3) Condition yields that $L_K(E)$ is prime.  Using Lemma \ref{LRSLemmas}, we embed $L_K(E)$ as a two-sided ideal in a prime $K$-algebra $L_K(E)_1$, and establish the primitivity of $L_K(E)$ by establishing the primitivity of $L_K(E)_1$.    Let $v$ be any vertex in $E$, and let $T(v) = \{u\in E^0 \ | \ v\geq u\}$.  Since $E$ is row-finite, the set $T(v)$ is at most countable. So we may label the elements of $T(v)$ as $\{v_1, v_2, ... \}$.
We
inductively define a sequence  $\lambda_1, \lambda_2, ...$ of paths in $E$ for which, for each $i\in \mathbb{N}$,
\begin{enumerate}
 \item $\lambda_i$ is an initial subpath of $\lambda_j$ whenever $i\leq j$, and
       \item  $v_i \geq r(\lambda_i)$.
           \end{enumerate}

To do so, define $\lambda_1 = v_{1}$.  Now suppose $\lambda_1, ..., \lambda_n$ have been defined with the indicated properties for some $n\in \mathbb{N}$.   By Condition (MT3), there is
a vertex $u_{n+1}$ in $E$ for which $r(\lambda_n) \geq u_{n+1}$ and $v_{n+1}\geq u_{n+1}$.
 Let $p_{n+1}$ be a path from $r(\lambda_n)$ to $u_{n+1}$, and define $\lambda_{n+1}=\lambda_{n}p_{n+1}$.    Then $\lambda_{n+1}$ is clearly seen to have the desired properties.

 Since $\lambda_i$ is a subpath of $\lambda_t$ for all $t\geq i$, we get that
 $$\lambda_i\lambda_i^{\ast}\lambda_t\lambda_t^{\ast} = \lambda_t\lambda_t^{\ast} \ \ \mbox{for each pair of positive integers } \ t\geq i.$$
 Now define the left $L_K(E)_1$-ideal $M$ by setting $$M=\sum\limits_{i=1}^{\infty }L_K(E)_1(1-\lambda_{i}\lambda_{i}^{\ast }).$$
 We first claim that $M\neq L_K(E)_1$.  On the contrary,  suppose $1 \in M$. Then there exists $n\in \mathbb{N}$ and $r_1,...,r_n \in L_K(E)_1$ for which  $1 = \sum_{i=1}^n r_i(1-\lambda_i \lambda_i^{\ast})$.  Multiplying this equation on the right by $\lambda_n \lambda_n^{\ast}$, and using the displayed observation, yields $\lambda_n \lambda_n^{\ast} = 0$.  But this is impossible, since it would imply $0 = \lambda_n^{\ast} \lambda_n \lambda_n^{\ast} \lambda_n = r(\lambda_n)$.  Thus $M$ is indeed a proper left ideal of $L_K(E)_1$.

 We now verify
that $M$ is comaximal with every nonzero two-sided ideal $I$ of $L_K(E)_1$.    Since $L_K(E)_1$ is prime and $L_K(E)$ embeds in $L_K(E)_1$ as a two-sided ideal,  we have $I\cap L_K(E)$ is a nonzero two-sided ideal of $L_K(E)$.    So   Condition
(L) on $E$, together with \cite[Corollary 3.3(ii)]{ABGM}, implies that $I$ contains some vertex, call it $w$.  By Condition (MT3) there exists $u\in E^0$ for which $v\geq u$ and $w\geq u$.  But $v\geq u$ gives by definition that $u=v_n$ for some $n\in \mathbb{N}$, so that $w\geq v_n$.   By the construction of the indicated sequence of paths $v_n \geq r(\lambda_n)$, so that there is a path $q$ in $E$ for which $s(q)=w$ and $r(q)=r(\lambda_n)$.  Since $w\in I$ this gives $r(\lambda_n) \in I$, so that $\lambda_n\lambda_n^{\ast} = \lambda_n r(\lambda_n) \lambda_n^{\ast} \in I$.  But then $1 = (1 - \lambda_n\lambda_n^{\ast}) + \lambda_n \lambda_n^{\ast} \in M + I$, so that $M+I = R$ as desired.

Thus $M$ satisfies the  conditions of Lemma \ref{primitivitylemma}, which yields the left primitivity of $L_K(E)_1$, and with it, by Lemma \ref{LRSLemmas}, the primitivity of $L_K(E)$.

\smallskip

Conversely, suppose $R=L_{K}(E)$ is  primitive. Since $R$
is then in particular a prime ring, $E$ has Condition (MT3)  by Theorem \ref{primenesstheorem}.  We argue by contradiction that $E$ has Condition (L) as well, for, if not, there is a cycle $c$
based at a vertex $v$ in $E$ having no exits. But then   by \cite[Lemma 1.5]{ABGM}, the corner ring $vRv\cong K[x,x^{-1}]$
 is not primitive (as a commutative primitive ring must be a field). Since a corner of a primitive ring must again
be primitive (for if $M$ is a faithful simple left $R$-module, then $eM$ is a faithful simple left $eRe$-module), we reach the desired  contradiction, and the result follows.
\end{proof}

\begin{rem}
{\rm The result appearing directly above as Theorem \ref{primitiverowfinite} in fact appears as Theorem 4.6 of \cite{APPSM2}.   However, the the proof of \cite[Theorem 4.6]{APPSM2} relies on \cite[Propositions 4.4 and 4.5]{APPSM2}, and these latter two results are in fact incorrect.   Briefly paraphrased, the statement of \cite[Proposition 4.4(i)]{APPSM2} asserts that if $u\geq v$ are vertices, then the sets $S_u$ and $S_v$ of simple factors of $uL_K(E)$ and $vL_K(E)$ (respectively) are equal, and, moreover, \cite[Proposition 4.4(ii)]{APPSM2} claims that the simple factor $uL_K(E)/J$ is isomorphic to $vL_K(E)/\alpha^*J$ (where $\alpha$ is a path from $u$ to $v$).   Furthermore,  \cite[Proposition 4.5]{APPSM2} proposes that if $u$ is a vertex which emits at least two edges, and $M$ is a simple factor of the right ideal $uL_K(E)$, then for some edge $e$ emitted by $u$ there is an isomorphism $M\cong r(e)L_K(E)/e^*L_K(E)$.    But consider the graph
$$
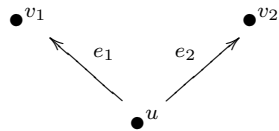
\xymatrix{{\bullet}^{v_1}  & & {\bullet}^{v_2} \\
 & {\bullet}^{u} \ar[ul]_{e_1} \ar[ur]^{e_2} & } $$
For each $i=1,2$ it is not hard to show that $v_i = r(e_i)$ has the property that $v_iL_K(E)$ is a simple right $L_K(E)$-module, and, furthermore,  that $e_i^*L_K(E) = v_iL_K(E)$.  Thus the quotients $r(e_i)L_K(E)/e_i^*L_K(E)$ are each zero.  This clearly violates both of the statements of \cite[Proposition 4.4(ii) and 4.5]{APPSM2}.  Furthermore, in this case we have $| \ S_u \ | = 2$, while $| \ S_{v_i} \ | = 1$ for $i=1,2$, thus violating  \cite[Proposition 4.4(i)]{APPSM2}.
}
\end{rem}

\section{Primitive Leavitt path algebras for arbitrary graphs: \\ sufficient conditions}

In this section we present three conditions on an arbitrary graph $E$ which, taken together,  imply the primitivity of  the Leavitt path algebra $L_K(E)$  for any field $K$.  We start by defining  a property of graphs which will play a key role in determining the primitivity of Leavitt path algebras of not necessarily row-finite graphs.   As motivation for this property, we observe that the proof of the sufficiency portion of the row-finite case (Theorem \ref{primitiverowfinite})  reveals that,   for any $v\in E^0$,  the set $T(v)$ is an at most countable subset of $E^0$ such that, for every $w\in E^0$, there is some $v_i \in T(v)$ for which $w \geq v_i$.

\begin{defn} {\em Let $E$ be a  graph and let $S\subseteq E^0$.  We say that $S$ has the \emph{Countable Separation Property} (more concisely, \emph{CSP}) if there is an at most countable set of vertices $C(S) = \{v_1,v_2,\ldots\} \subseteq E^0$ such that, for every $w\in S$, there is some $v_i\in C(S)$ for which $w \geq v_i$.    \hfill $\Box$ }
\end{defn}

\begin{rem}\label{CSPforcountablegraphremark}
{\rm Let $E$ be a graph for which $E^0$ is at most countable.  Then every subset $S$ of $E^0$ has CSP (simply use $C(S) = E^0$). \hfill $\Box$

}
\end{rem}

\begin{lem}\label{rowfiniteMT3givesCSPlemma}
Let $E$ be a row-finite graph which satisfies Condition (MT3).  Then $E^0$ has CSP.
\end{lem}
\begin{proof}  For any $v\in E^0$, consider the set $T(v) = \{u\in E^0 \ | \ v\geq u\}$.  As $E$ is row-finite, $T(v)$ is at most countable.  But then the (MT3) condition yields that $T(v)=C(E^0)$ has the desired property.
\end{proof}

\begin{defn}\label{EsubTsubXdefinition}
{\rm Let $X$ be any nonempty set, and let $\mathcal{F}(X)$ denote the collection  of  nonempty finite  subsets of $X$.   We  define the  graph $E_{\mathcal{F}(X)}$ as follows:
$$E_{\mathcal{F}(X)}^0 \ = \ \mathcal{F}(X), \ \ \
E_{\mathcal{F}(X)}^1 = \{e_{A,A'} \ | \ A,A'\in \mathcal{F}(X), \ \mbox{and} \ A\subsetneqq A'\},$$
  $s(e_{A,A'})=A$, and $r(e_{A,A'})=A'$ for each $e_{A,A'} \in E_{\mathcal{F}(X)}^1.$
 \hfill  $\Box$  }
\end{defn}

 Using  Remark \ref{CSPforcountablegraphremark} along with standard cardinal arithmetic, we immediately get

\begin{prop}\label{EsubTsubXhasCSPiffXcountableprop}
Let $X$ be any set, and let $E_{\mathcal{F}(X)}$ be the graph described in Definition \ref{EsubTsubXdefinition}.  Then $E_{\mathcal{F}(X)}^0$ has CSP if and only if $X$ is at most countable.
\end{prop}

For instance,  if $X = \mathbb{R}$, then $\mathcal{F}(X)$ consists of the finite nonempty subsets of $\mathbb{R}$, and the resulting graph $E_{\mathcal{F}(\mathbb{R})}$ does not have the Countable Separation Property (as no countable union of finite subsets of $\mathbb{R}$ can contain all of $\mathbb{R}$.)

\begin{defn}\label{Esubkappadefinition}
{\rm Let $\kappa$ be any ordinal.  We  define the graph $E_{\kappa}$ as follows:
$$E_{\kappa}^0 \ = \{\alpha \ | \ \alpha < \kappa\}, \ \ \ E_{\kappa}^1 = \{e_{\alpha,\beta} \ | \ \alpha, \beta < \kappa, \ \mbox{and} \ \alpha < \beta\},$$
  $s(e_{\alpha,\beta})=\alpha$, and $r(e_{\alpha,\beta})=\beta$ for each $e_{\alpha,\beta} \in E_{\kappa}^1.$
 \hfill  $\Box$  }
\end{defn}

Recall that an ordinal $\kappa$ is said to have {\it countable cofinality} in case $\kappa$ is the limit of a countable sequence of ordinals strictly less than $\kappa$.   For example, any countable ordinal has countable cofinality.  The ordinal $\omega_1$ does not have countable cofinality, while the ordinal $\omega_\omega$ does have this property.   With this definition in mind, the following result is clear.

\begin{prop}\label{EsubkappahasCSPiffkappacountablecofinalityprop}
Let $\kappa$ be any ordinal, and let $E_{\kappa}$ be the graph described in Definition \ref{Esubkappadefinition}.  Then $E_{\kappa}^0$ has CSP if and only if $\kappa$ has countable cofinality.
\end{prop}

The Countable Separation Property of $E^0$ will turn out to be a necessary ingredient in the primitivity of $L_K(E)$ for arbitrary graphs $E$.   Indeed one can show, using an argument similar to that employed in the proof of Theorem \ref{primitiverowfinite}, that CSP together with Conditions (MT3) and (L) are sufficient for primitivity.   However, we choose to give an alternate, somewhat more streamlined proof here, utilizing the following result of Fisher and Snider.

 \begin{thm}\label{FisherSniderTheorem}   \cite[Theorem 1.1]{FS} \  Suppose $R$ is a prime ring in which each nonzero right ideal contains a nonzero idempotent.  Suppose also that $R$ contains a countable set $\mathcal{I}$ of nonzero ideals, with the property that each nonzero ideal of $R$ contains an element of $\mathcal{I}$.  Then $R$ is both left and right primitive.
\end{thm}

It was established in \cite[Corollary 3.3]{ABGM} that if the graph $E$ satisfies Condition (L), then every nonzero two-sided ideal of $L_K(E)$ contains a nonzero idempotent, specifically, a  vertex.  The following is a generalization of this result to one-sided ideals.  We thank Enrique Pardo for the proof.

\begin{lem}\label{onesidedidealcontainsvertex}
Let $E$ be a graph which satisfies Condition (L), and let $K$ be any field.  Let $N$ be a nonzero right (or left) ideal of $L_K(E)$.  Then there exists a nonzero idempotent $e\in N$.
\end{lem}
\begin{proof}
Let $a$ be a nonzero element of the right ideal $N$ of $L_K(E)$. By \cite[Proposition 3.1]{ABGM} there exists $x,y\in L_K(E)$ and $v\in E^0$ for which  $xay=v$.  By multiplying both sides by $v$ on the left and right, we may assume  $x=vx$ and $y=yv$. Now define $e:=ayx$; since $a\in N$ then $e\in N$.  In addition, $e\neq 0$ since $0 = e = ayx$ would yield $0 = xayx = vx =x$, so that $v=xay = 0$, a contradiction.   Finally,
$$ e^2=(ayx)(ayx)=ay(xay)x =ayvx=ayx=e,$$
which establishes the result for right ideals.
That the result is true for left ideals as well follows immediately from Proposition \ref{L(E)isoL(E)op}.
\end{proof}

We are now in position to prove the sufficiency result.

\begin{prop}\label{threeconditionsimplyprimitiveprop}  Let $E$ be an arbitrary graph, and $K$ any field.   Suppose:
\begin{enumerate}[(i)]
  \item  $E$ satisfies Condition (MT3),
\item  $E$ satisfies Condition (L), and
\item  $E^0$ has the Countable Separation Property.
\end{enumerate}
Then $L_K(E)$ is primitive.
\end{prop}

\begin{proof}
Since $E$ satisfies Condition (MT3), Theorem \ref{primenesstheorem} yields that $L_K(E)$ is prime.  By Condition (L) and Lemma \ref{onesidedidealcontainsvertex}, every nonzero right ideal of $L_K(E)$ contains a nonzero idempotent.

Now let $T = \{v_i \ | \ i\in \mathbb{N}\}$ be a set of vertices with respect to which $E^0$ has the Countable Separation Property, and let $B$ be any nonzero (two-sided) ideal of $L_K(E)$.  Again using Condition (L),  \cite[Corollary 3.3]{ABGM}  yields that there exists $v\in E^0 \cap B$.  By hypothesis, there exists $v_i \in T$ for which $v\geq v_i$.  But then $v_i \in B$, so that $L_K(E)v_iL_K(E) \subseteq B$.   Thus the set
$$\mathcal{I} \ = \ \{L_K(E)v_iL_K(E) \ | \ i\in \mathbb{N}\}$$
is a countable set of nonzero ideals of $L_K(E)$ having the property that each nonzero two-sided ideal of $L_K(E)$ contains an ideal in this set.

So we have established that the hypotheses of Theorem \ref{FisherSniderTheorem} are satisfied in $L_K(E)$, which yields the result.
\end{proof}

As mentioned in Section \ref{rowfinitesection}, Theorem \ref{primitiverowfinite} may also be established as a consequence of  Theorem \ref{threeconditionsimplyprimitiveprop}.  Specifically, by Lemma \ref{rowfiniteMT3givesCSPlemma}, CSP is automatically satisfied in the row-finite case, so that by Theorem \ref{threeconditionsimplyprimitiveprop}, Conditions (MT3) and (L) are sufficient in that case to imply the primitivity of $L_K(E)$.  That Conditions (MT3) and (L) are necessary was established in the proof of Theorem \ref{primitiverowfinite}.

Similarly, using Remark \ref{CSPforcountablegraphremark}, we get as well  the following consequence of Proposition \ref{threeconditionsimplyprimitiveprop}.

\begin{cor}
Let $E$ be a graph for which $E^0$ is at most countable, and let  $K$ be any field.  Then $R=L_{K}(E)$ is  primitive
 if and only if

(i) $E$ satisfies Condition (MT3), and

(ii) $E$ satisfies Condition (L).

\end{cor}

\bigskip

We conclude this section by offering a few observations regarding the ``prime implies primitive" question.

\begin{prop}\label{generalprimitivityprop}
Let $R$ be a unital ring.  Suppose $R$ contains a countable set of nonzero idempotents $\{e_i \ | \ i\in \mathbb{N}\}$ for which $e_ie_j = e_je_i = e_j$ for all $j\geq i$, and for which each nonzero ideal of $R$ contains some $e_i$.    Then $R$ is primitive.
\end{prop}

 \begin{proof}
 The argument  mimics exactly  the first half of the proof of Theorem \ref{primitiverowfinite}:  one defines the left $R$-ideal $M = \sum_{i=1}^\infty R(1-e_i)$, and shows that the hypotheses on the set $\{e_i \ | \ i\in \mathbb{N}\}$ yield that $M$ satisfies the properties of Lemma \ref{primitivitylemma}.
 \end{proof}

Essentially,  the proof of Theorem \ref{FisherSniderTheorem} is accomplished by showing that the given hypotheses lead to a set of idempotents of the type described in Proposition \ref{generalprimitivityprop}, and then concluding thereby that $R$ is primitive.

We may use a similar approach in the context of Leavitt path algebras.  Suppose $E^0$ has CSP along with Conditions (MT3) and (L).  Then by a construction nearly identical to the one presented in Theorem \ref{primitiverowfinite} (replacing $T(v)$ by the supposed countable set $C(E^0)$), we may build a countable set of nonzero idempotents $\{ \lambda_i \lambda_i^\ast \ | \ i\in \mathbb{N} \}$ in $L_K(E)$ which satisfies  the two conditions of Proposition \ref{generalprimitivityprop}, and thus the primitivity of $L_K(E)$ follows.   (This approach does utilize the fact that by Condition (L) every nonzero two-sided ideal of $L_K(E)$ contains a vertex, but does not utilize Lemma \ref{onesidedidealcontainsvertex}.)

\smallskip


\section{Primitive Leavitt path algebras for arbitrary graphs: \\ The necessity of the Countable Separation Property of $E^0$ }

We now establish the converse of Proposition \ref{threeconditionsimplyprimitiveprop}.
 The heavy lifting needed to achieve this goal will come in showing that if $E$ does not have CSP, then $L_K(E)$ is not primitive.

\begin{rem}
\label{CSPresultsRemark}
{\rm  Along the way we shall use (most often without explicit mention) various characteristics of the Countable Separation Property, each of which is easy to verify.  We collect those up in the following list.
\begin{enumerate}

 \item  The empty set has CSP.

 \item  If $X\subseteq Y$ and $Y$ has CSP, then $X$ has CSP.  \ (Equivalently, if $X\subseteq Y$ and $X$ does not have  CSP, then $Y$ does not  have CSP.)

 \item If $S\subseteq E^0$ does not have CSP and $S_1,S_2,\ldots \subseteq S$ are countably many subsets of $S$, each of which has CSP, then  $S \ \setminus \  \cup_{i=1}^{\infty} S_i$ does not have CSP.

     \item  If $S$ is the union of a countable number of subsets of $E^0$ each of which has CSP, then $S$ has CSP.

 \item If $S \subseteq \cup_{i=1}^n X_i$ and $S$ does not have CSP, then $X_i$ does not have CSP for some $1\leq i \leq n$.

     \item  If $S \subseteq Y\cup Z$ and $Y$ has CSP, then $X$ does not have CSP if and only if $Z$ does not have CSP.  \hfill $\Box$
\end{enumerate}
}
\end{rem}

Our approach is as follows.  Let $R$ denote $L_K(E)$, assume that $R$ is prime, and let $R_1$ denote a prime unital ring into which $R$ embeds as a two-sided ideal.       We show that, in the absence of the Countable Separation Property on $E^0$, the collection $\{R_1vR_1 \ | \ v\in E^0\}$ of two-sided ideals of $R_1$ is ``sparse", in the sense that if all of the ideals in this collection are comaximal with a left ideal $I$ of $R_1$, then necessarily $I=R_1$.  This conclusion, combined with Lemma \ref{primitivitylemma}, will show that $R_1$, and thereby $R$, is not primitive.  We utilize the following simple but important observation to draw such a conclusion.
\begin{lem}
Let $I$ be a left ideal of a unital ring $A$.  If there exist $x,y\in A$ such that $1+x \in I$, $1+y \in I$, and $xy=0,$ then $I=A$.
\label{lem: zeromult}
\end{lem}
\begin{proof}
Since $1+y\in I$ then $x(1+y)=x+xy=x \in I$, so that $1 = (1+x)-x \in I$.
\end{proof}

\begin{lem} \label{lem: fund}
Let $E$ be a directed graph and let $T$ be a subset of $E^0$  which does not have the Countable Separation Property.  Suppose that we can associate with each $v\in T$  an element
$x_v\in L_K(E)$
in such a way that, for each $v\in T$, the set
$$Z(T,v) \ = \ \{w\in T~|~x_vx_w=0\}$$
 has CSP.  Then the set
$$ T^\sim \  = \ \{v\in T \ | \ \{w\in T \ | \  x_vx_w\neq 0\} \ \mbox{does not have CSP}\}$$ does not have CSP.
\end{lem}
\begin{proof}  Suppose to the contrary that  $T^\sim$ has CSP.    Then $T':=T\setminus T^\sim$ does not have CSP, and so in particular is nonempty.
 Now fix  $v\in T'$.  Then clearly
 $$T' \ = \ \{w\in T'~|~x_vx_w\neq 0\} \ \cup \ \{w\in T'~|~x_vx_w=0\}.$$
 As $v \notin T^\sim$, the first  displayed set has CSP.  So, since $T'$ does not have CSP, the second displayed set does not have CSP.  But this contradicts the hypothesis that $Z(T,v)$ has CSP for all $v\in T$.
\end{proof}

The following result will yield the fundamental relationship between the Countable Separation Property in $E^0$ on the one hand,  and the form of various sets of elements of $L_K(E)$ on the other.

\begin{prop} \label{lackofCSPProp}
 Let $E$ be an arbitrary graph, and $K$ any field.  Let $S$ be a subset of $E^0$ which does not have CSP.  Let $d$ be a fixed positive integer, and let  $(m_1,\ldots ,m_d)$ and $(n_1,\ldots ,n_d)$ be a fixed pair of  sequences of nonnegative integers of length $d$.  Let $[d]$ denote the set $\{1,2,...,d\}$.    Suppose we can associate with each element $v\in S$ an element $x_v$ of $L_K(E)$ of the form
$$x_v=\sum_{i=1}^d k_{i,v} a_{i,v} b_{i,v}^*$$
for which, for all $v\in S$ and $i \in [d]$:
\begin{enumerate}[(i)]
\item  $k_{i,v} \in K$,
\item  the length $\ell(a_{i,v})$ of $a_{i,v}$ is $m_i$,
\item  the length $\ell(b_{i,v})$ of $b_{i,v}$ is $n_i$, \ and
\item $v\geq r(a_{i,v})$  \ (so that  $v\geq r(b_{i,v})$ as well).
\end{enumerate}
Then there exists $v$ in $S$ for which the set
$$Z(S,v) \ = \ \{w\in S~|~x_vx_w=0\} $$
does not have CSP.
\end{prop}

\begin{proof} Suppose not.  Then for each $v\in S$, we have that the set
$Z(S,v) = \{w\in S | x_vx_w=0\}$ has CSP.   Indeed, then,  this same property passes to any subset $T\subseteq S$; that is, if $T\subseteq S$, then we have that for each $v\in T$ that the set
$$Z(T,v) = \{w\in T  \ | \  x_vx_w=0\}$$
 has CSP. If in addition the subset $T$ of $S$ does not have CSP, then Lemma \ref{lem: fund} applies to $T$, so that the set
$$T^\sim \  = \ \{v\in T~|~\{w\in T~|~x_vx_w\neq 0\}~{\rm does~not~have~CSP}\}$$
does not have CSP.
\medskip

Fix $v\in T^\sim$.   Now $x_vx_w\neq 0$ implies that there exists at least one pair $(i,j)\in [d]^2$ for which $b_{i,v}^\ast a_{j,w} \neq 0$, so that
 $$\{w\in T \ | \ x_vx_w\neq 0\} = \bigcup_{(i,j)\in [d]^2} \{w\in T \ | \ x_vx_w\neq 0 \ \mbox{and} \ b_{i,v}^*a_{j,w} \neq 0\}.$$
Thus $$ T^\sim  \subseteq \bigcup_{(i,j)\in [d]^2} \{v\in T \ | \ \{w\in T \ | \ x_vx_w\neq 0 \ \mbox{and} \ b_{i,v}^*a_{j,w} \neq 0\} \ \mbox{does not have CSP} \},$$
  a union of finite sets.  Then by Remark \ref{CSPresultsRemark} we have that there exists $(i,j)\in [d]^2$ for which
$$\{v\in T \ | \ \{w\in T \ | \ x_vx_w\neq 0 \ \mbox{and} \ b_{i,v}^*a_{j,w} \neq 0\}  \ \mbox{does not have CSP} \}$$
does not have CSP.
To summarize:   assuming that the set $T^\sim$ does not have CSP, we are led to  an element of $[d]^2$, which we denote by $(i_T,j_T)$,  for which
$$\{v\in T \ | \ \{w\in T \ | \ x_vx_w\neq 0 \ \mbox{and} \ b_{i_T,v}^*a_{j_T,w} \neq 0\}  \ \mbox{does not have CSP} \} $$
does not have CSP.
\medskip

Now let $\mathcal{S}$ denote the collection of subsets $T$ of the original set $S$
for which $T$ does not have CSP.  Then the above discussion yields that there exists a (not necessarily unique) map
$$\pi : \mathcal{S} \to  [d]^2,$$ where
$\pi(T)=(i_T,j_T)$ in case $$\{v\in T \ |   \ \{w\in T \ | \  x_vx_w\neq 0 \ \mbox{and} \   ~b_{i_T,v}^*a_{j_T,w}\neq 0\}~{\rm does~not~have~CSP}\} \ \ \mbox{does not have CSP}.$$

\medskip

We now recursively define a countable collection of nested subsets $$S=S_0\supseteq S_1\supseteq S_2 \supseteq \cdots,$$ none of which have CSP.  We take $S_0=S$.  Then  $S_0$ does not have CSP by hypothesis.  Having defined $S_n$ for $n\ge 0$, we note that there exist $i_{S_n}$ and $j_{S_n}$ in $\{1,\ldots ,d\}$ such that
$\pi(S_n)=(i_{S_n},j_{S_n})$.
That is, the set
$$ (S_n)^\sim_{i_{S_n},j_{S_n}} \ = \ \{v\in S_n~|~\{w\in S_n~| \  x_vx_w \neq 0 \ \mbox{and} \ ~b_{i_{S_n},v}^*a_{j_{S_n},w}\neq 0\}~{\rm does~not~have~CSP}\} $$
does not have CSP. In this case, we pick any $v\in (S_n)^\sim_{i_{S_n},j_{S_n}}$; that is, pick any $v\in S_n$ for which
$$\{w\in S_n~| \  x_vx_w \neq 0, \ \mbox{and} \ ~b_{i_{S_n},v}^*a_{j_{S_n},w}\neq 0\} \ \mbox{does not have CSP}.$$   (Such $v \in S_n$ exists since the set $(S_n)^\sim_{i_{S_n},j_{S_n}}$ of such $v$ does not have CSP, and thus is nonempty.)
If $b_{i_{S_n},v}^*a_{j_{S_n},w}\neq 0$ then either $b_{i_{S_n},v}$ is an initial subpath of $a_{j_{S_n},w}$, or $a_{j_{S_n},w}$ is an initial subpath of $b_{i_{S_n},v}$.
 Now define the set
 $$X  =\{w\in S_n~| \  x_vx_w\neq 0 \ \mbox{and} \ ~a_{j_{S_n},w}~{\rm is~an~initial~subpath~of ~}b_{i_{S_n},v}\}.$$
(Note that  $X$ depends on $n,v,i_{S_n}$, and $j_{S_n}$.)  We claim that $X$
  has CSP.  To see this, recall that, by hypothesis,
  $w \geq r(a_{j_{S_n},w})$.
    But  $r(a_{j_{S_n},w}) \geq r(b_{i_{S_n},v})$
     since $a_{j_{S_n},w}$
      is an initial subpath of $b_{i_{S_n},v}$.
        So we have that $w\geq r(b_{i_{S_n},v})$
         for all $w\in X$, which yields that $X$
          has CSP with respect to the singleton set $\{r(b_{i_{S_n},v})\}$.

 Now $\{w\in S_n~| \  x_vx_w \neq 0 \ \mbox{and} \ ~b_{i_{S_n},v}^*a_{j_{S_n},w}\neq 0\}$ is the union of the two sets
$$ \{w\in S_n~| \ x_vx_w \neq 0 \ \mbox{and} \ ~b_{i_{S_n},v} \ \mbox{is an initial subpath of } \ a_{j_{S_n},w} \}$$
 $$\cup \ \{w\in S_n~| \  x_vx_w \neq 0 \ \mbox{and} \ a_{j_{S_n},w} \ \mbox{is an initial subpath of } \  ~b_{i_{S_n},v}\}.$$
Since the second displayed set (this is the set $X$)  was demonstrated in the previous paragraph to have CSP, we get that, for any $v\in (S_n)^\sim_{i_{S_n},j_{S_n}}$,  the set
$$\{w\in S_n~| \ x_vx_w \neq 0, \ \mbox{and} \ ~b_{i_{S_n},v}^*a_{j_{S_n},w}\neq 0\}$$
 does not have CSP if and only if the set
 $$S_n'(v) = \{w\in S_n~| \ x_vx_w \neq 0 \ \mbox{and} \ ~b_{i_{S_n},v} \ \mbox{is an initial subpath of } \ a_{j_{S_n},w}\}$$
 does not have CSP.

 We record for later use that in particular this yields that the set $(S_n)^\sim_{i_{S_n},j_{S_n}}$ defined previously in fact equals the set
 $$S_n' \ = \ \{v\in S_n \ | \  \{w\in S_n~| \ x_vx_w \neq 0 \ \mbox{and} \ ~b_{i_{S_n},v} \ \mbox{is an initial subpath of } \ a_{j_{S_n},w}\} \ \mbox{does not have CSP} \},$$
 and so $S_n'$ does not have CSP, and so in particular is nonempty.

Now pick any $v_n\in (S_n)^\sim_{i_{S_n},j_{S_n}}$, and  define $S_{n+1}$ by setting
$$S_{n+1}  \ = \ S_n'(v_n) \ = \  \{w\in S_n~| \ x_{v_n}x_w \neq 0, \ \mbox{and} \ ~b_{i_{S_n},v_n} \ \mbox{is an initial subpath of } \ a_{j_{S_n},w}\}.$$
By definition, $S_{n+1}$ does not have CSP.   We note here the following:  for each $n+1$, there exists a  path $b_{n+1} = b_{i_{S_n},v_n}$ such that $a_{j_{S_n},w}$ has $b_{n+1}$ as an initial subpath for all $w\in S_{n+1}$.   Moreover, $\ell(b_{n+1}) = n_{i_{S_n}}$, and is in particular independent of the choice of $v_n$.

\medskip

Thus we have defined the sequence
$$S = S_0 \supseteq S_1 \supseteq S_2 \supseteq \cdots$$
in such a way that no $S_n$ has CSP.

\vskip 2mm
Since the range of $\pi$ is finite, we see that there exist natural numbers $f$ and $g$ with $f<g$ such that $\pi(S_f)=\pi(S_g)=(i,j)$ for some pair $(i,j)$ in $[d]^2$.   We show that this leads to a contradiction, which will establish the result.

So we have $i_{S_f} = i_{S_g} = i$, and $j_{S_f} = j_{S_g} = j$.  By the above discussion,  there exists a path $b_{f+1}$ of length $n_i$ for which $b_{f+1}$ is an initial subpath of $a_{j,w}$ for all $w\in S_{f+1}$.

Since $S_{f+1} \supseteq S_{g+1}$, this property holds for all elements of $S_{g+1}$ as well.     But we know that there exists a path $b_{g+1}$ of length $n_i$ (n.b.: the {\it same} $n_i$ as in the previous paragraph) for which $b_{g+1}$ is an initial subpath of
$a_{j,w}$
 for all $w\in S_{g+1}$.
So, since the lengths of each of $b_{g+1}$ and $b_{f+1}$ are equal (to $n_i$), and each is the initial subpath of a common path (any $a_{j,w}$ for $w\in S_{g+1}$), we have that the paths  $b_{g+1}$ and $b_{f+1}$ are in fact equal; we denote this common path by $b$.

Recall that the set $S_g'$ is defined above as
$$S_g' \ = \ \{v\in S_g \ | \  \{w\in S_g~| \ x_vx_w \neq 0, \ \mbox{and} \ ~b_{i_{S_g},v} \ \mbox{is an initial subpath of } \ a_{j_{S_g},w}\} \ \mbox{does not have CSP} \},$$
which in this situation can be rewritten as
$$S_g' \ = \ \{v\in S_g \ | \  \{w\in S_g~| \  x_vx_w \neq 0, \ \mbox{and} \ ~b_{i,v} \ \mbox{is an initial subpath of } \ a_{j,w}\} \ \mbox{does not have CSP} \}.$$
Recall also that $S_g'$ does not have CSP.

We claim that
$$S_g' \subseteq \{v\in S_g \ | \  x_vx_w\neq 0 \ \mbox{and} \  b_{i,v} = b\}.$$   So suppose $v\in S_g'$; that is, suppose $v\in S_g$, and that the set
$$\{w\in S_g~| x_vx_w \neq 0, \ \mbox{and} \ ~b_{i,v} \ \mbox{is an initial subpath of } \ a_{j,w}\}$$
  does not have CSP.   We must show that $b_{i,v} = b$.  Since the latter set does not have CSP it is in particular nonempty, so let $w\in S_g$ have the property that  $x_vx_w \neq 0, \ \mbox{and} \ ~b_{i,v}$  is an initial subpath of  $ a_{j,w}$.   Since $w\in S_g\subseteq S_{f+1}$, and $j_{S_f} = j_{S_g} = j$, we have by a previous observation that $a_{j,w}$ has $b$ as an initial subpath.  But the length of $b_{i,v}$ is $n_i$ by definition, which equals the length of $b$, so that $b_{i,v} = b$, thus establishing the claim.

Since $v\geq r(b_{i,v})$ for all $v\in S$ by hypothesis, we see that the set $\{v\in S_g \ | \ x_vx_w\neq 0 \ \mbox{and} \  b_{i,v} = b\}$ has CSP with respect to the singleton set $\{r(b)\}$.    But this leads us to the desired contradiction,  since the last displayed inclusion would yield a set without CSP as a  subset of a set having CSP, which is impossible by Remark \ref{CSPresultsRemark}.
This contradiction completes the proof.
\end{proof}

The next result indicates that the conditions imposed in Proposition \ref{lackofCSPProp} are in fact quite natural in the context of two-sided ideals of $L_K(E)$ generated by vertices of $E$.

\begin{lem}\label{lemmaonelementsinRvR}
Let $E$ be an arbitrary graph, and $K$ any field. Let $v\in E^0$, and let $<v>$ denote the ideal $L_K(E)vL_K(E)$ of $L_K(E)$.   Then any $0\neq x \in \ <v>$ can be written as $x = \sum_{i=1}^n k_ia_ib_i^{\ast}$, where $k_i \in K$, $a_i, b_i$ are paths in $E$, $0\neq k_ia_ib_i^{\ast}$,  and $v \geq r(a_i) = r(b_i)$ for all $1\leq i \leq n$.
\end{lem}

\begin{proof}
By definition, $x$ is a $K$-linear combination of nonzero terms of the form $\alpha \beta^{\ast} v \gamma \delta^{\ast}$, where $\alpha, \beta, \gamma, \delta$ are paths in $E$.    Moreover, $\beta^{\ast} v \gamma \neq 0$ implies $v = s(\beta) = s(\gamma)$, and that either $\beta = \gamma\beta'$ or $\gamma = \beta \gamma'$ for some paths $\beta', \gamma'$ in $E$. With this, the expression $\alpha \beta^{\ast} v \gamma \delta^{\ast}$ simplifies to a term of the form $ab^{\ast}$, and it is readily seen that $v \geq r(a) = r(b)$.
\end{proof}

Here now is the key result regarding the non-primitivity of Leavitt path algebras.

\begin{prop} \label{notCSPimpliesnotprimitiveprop} Let $E$ be an arbitrary graph, and $K$ any field.  If  $E^0$ does not have the Countable Separation Property, then $L_K(E)$ is not primitive.
\end{prop}
\begin{proof}  If $L_K(E)$ is not prime then the result is immediate.  So we may suppose that $L_K(E)$ is prime.  Let $A$ denote the unital overring $L_K(E)_1$ of $L_K(E)$ as described in Lemma \ref{LRSLemmas}.  Just suppose $A$ is primitive; we seek a contradiction.   So by Lemma \ref{primitivitylemma} there is a proper left ideal $I$ of $A$ such that $I+J=A$ for every nonzero two-sided ideal $J$ of $A$.
Consequently, for each vertex $v\in E^0$, there exists some $x_v\in AvA$ such that $1+x_v\in I$.
Since $AvA\subseteq L_K(E)vL_K(E)$, we have  that each $x_v$ can be written as an element of the form indicated in Lemma  \ref{lemmaonelementsinRvR}.  With this in mind, we define
$$\Gamma \ = \ \bigcup_{d=1}^{\infty} \{d\}\times \left( \mathbb{Z}_{\ge 0}^{d}\right) \times  \left( \mathbb{Z}_{\ge 0}^{d}\right),$$ and for each $\gamma =(d,(m_1,m_2,\ldots ,m_d),(n_1,\ldots ,n_d))\in \Gamma$, we set
$$S_\gamma:=\{v\in E^0~|~x_v = \sum_{i=1}^d k_i a_ib_i^*,~k_i\in K, v\geq r(a_i)=r(b_i), \ell(a_i)=m_i, \ \mbox{and} ~\ell(b_i)=n_i~ \ \forall \  i\in [d] \}.$$
(We note that the $S_\gamma$ need not be disjoint.)  By Lemma \ref{lemmaonelementsinRvR} we have that
$\bigcup_{\gamma\in \Gamma} S_\gamma = E^0.$  But $\Gamma$ is countable, so since by hypothesis $E^0$ does not have CSP we get
 that there is some $\gamma\in \Gamma$ such that $S_\gamma$ does not have CSP.  We now apply  Proposition \ref{lackofCSPProp} to the set  $S_\gamma$ to conclude that there exists $v\in S_\gamma$ such that
$$ Z(S_\gamma,v)  =\{w\in S_\gamma~|~x_vx_w=0\}$$ does not have CSP.   In particular $Z(S_\gamma,v)$ is nonempty, so choose some $w\in Z(S_\gamma,v)$.  Then the elements $x_v$ and $x_w$ of $L_K(E)$ corresponding to the vertices $v$ and $w$ respectively have the property that
$$1+x_v \in I, \ 1+x_w \in I, \ \mbox{and} \ x_vx_w = 0.$$ Thus by Lemma \ref{lem: zeromult}, we see that $I=A$, a contradiction.

 Thus $A$ is not primitive, and so $L_K(E)$ is not primitive by Lemma \ref{LRSLemmas}.
\end{proof}

We now have all the pieces in place to establish our main result.

\begin{thm}\label{primitivitytheorem}
Let $E$ be an arbitrary graph, and $K$ any field.  Then $L_K(E)$ is primitive if and only if
\begin{enumerate}[(i)]
\item $E$ satisfies Condition (MT3),
\item $E$ satisfies Condition (L), and
\item $E^0$ has the Countable Separation Property.
\end{enumerate}
\end{thm}

\begin{proof}
That the three conditions are sufficient is established in Proposition \ref{threeconditionsimplyprimitiveprop}, while the necessity of Conditions (MT3) and (L) was established in the proof of Theorem \ref{primitiverowfinite} (the row-finiteness hypothesis is not used in those portions of that proof).

 The final piece of the proof of Theorem \ref{primitivitytheorem} (i.e., the necessity of Countable Separation Property) now follows directly from Proposition \ref{notCSPimpliesnotprimitiveprop}.
\end{proof}

\section{Prime non-primitive von Neumann regular rings}\label{primenotprimitivevNrsection}

As mentioned in the Introduction, for many years Kaplansky's  query  ``{\it Is a regular prime ring necessarily primitive?}"  was regarded as the major question in the theory of von Neumann regular rings.   Theorem \ref{primitivitytheorem} will allow us to identify various infinite classes of  rings for which the answer to Kaplansky's question is {\it No}.  We begin by noting a structural property of the Leavitt path algebras defined previously.

\begin{lem}\label{cardinalitylemma}  Let $K$ be any field.

  \begin{enumerate}
  \item Let $X$ be an infinite set, let  $|X|$ denote the cardinality of $X$, and let $E_{\mathcal{F}(X)}$ be the graph described in Definition \ref{EsubTsubXdefinition}.  Then ${\rm dim}_K(L_K(E_{\mathcal{F}(X)})) = |X|$.

  \item Let $\kappa$ be any infinite ordinal, let $|\kappa|$ denote the cardinality of $\kappa$, and let $E_{\kappa}$ be the graph described in Definition \ref{Esubkappadefinition}.  Then ${\rm dim}_K(L_K(E_\kappa)) = |\kappa|$.
\end{enumerate}
\end{lem}

\begin{proof}
(1)   Since $|X|$ is infinite, the cardinality of the set $\mathcal{F}(X)$ of finite subsets of $X$ is  $|X|$.  Thus the graph $E_{\mathcal{F}(X)}$ contains $|X|$ vertices.  This set of vertices, being nonzero orthogonal idempotents of $L_K(E_{\mathcal{F}(X)})$, are  necessarily $K$-linearly independent.  So ${\rm dim}_K(L_K(E_{\mathcal{F}(X)})) \geq |X|$.

On the other hand, each vertex of $E_{\mathcal{F}(X)}$ emits $|X|$ edges.  Since by definition a path in $E$ is a finite sequence of edges, and there are $|X|$ vertices in $E_{\mathcal{F}(X)}$, this yields that there are $|X|$ distinct paths in $E$, so that there are $|X|$ expressions of the form $pq^*$ in $L_K(E)$ (where $p$ and $q$ are paths in $E$).  But for any Leavitt path algebra, the set $\{pq^* \ | \ p,q  $ are paths in $E\}$ spans $L_K(E)$ as a $K$-vector space.   Thus ${\rm dim}_K(L_K(E_{\mathcal{F}(X)})) \leq |X|$, which yields the result.

\smallskip

The proof of (2) is similar.
\end{proof}

\begin{prop}\label{primenotprimitivecor}
Let $K$ be any field.  Let $X$ be any nonempty set, and let $E_{\mathcal{F}(X)}$ be the graph presented in Definition \ref{EsubTsubXdefinition}.
 Then $L_K(E_{\mathcal{F}(X)})$ is a prime ring.  Moreover, $L_K(E_{\mathcal{F}(X)})$  is primitive if and only if $X$ is at most countable.


\end{prop}

\begin{proof}
It is clear that $E_{\mathcal{F}(X)}$ is acyclic,  and thereby satisfies Condition (L) vacuously. Since the union of two finite sets is finite,    Condition (MT3) holds in $E_{\mathcal{F}(X)}$ as well.  In particular, $L_K(E_{\mathcal{F}(X)})$ is prime by Theorem \ref{primenesstheorem}.    But by Proposition \ref{EsubTsubXhasCSPiffXcountableprop}, $E^0$ has CSP precisely when $X$ is at most countable.  The result now follows from  Theorem \ref{primitivitytheorem}.
\end{proof}

Here is the first of the previously mentioned classes of prime, non-primitive, von Neumann regular algebras.

\begin{thm}\label{primenotprimitivevNrthmusingEsubTsubX}
Let $K$ be any field.  Let $X$ be an uncountable set, and let $E_{\mathcal{F}(X)}$ be the graph presented in Definition \ref{EsubTsubXdefinition}.
 Then $L_K(E_{\mathcal{F}(X)})$ is a prime, non-primitive, von Neumann regular ring.

 In particular, let $\{X_\alpha  \ | \  \alpha \in \mathcal{A}\}$ denote a collection of uncountable sets, each of different cardinality.   Then the collection $\{L_\mathbb{C}(E_{\mathcal{F}({X_\alpha})}) \ | \ \alpha \in \mathcal{A}\}$ consists of pairwise nonisomorphic, prime, non-primitive, von Neumann regular $\mathbb{C}$-algebras.
\end{thm}

\begin{proof}
It was established in \cite[Theorem 1]{AR} that the von Neumann regular Leavitt path algebras are precisely those corresponding to acyclic graphs. In particular, each $L_\mathbb{C}(E_{\mathcal{F}({X_\alpha)}})$ is von Neumann regular.  The first statement now follows from Proposition  \ref{primenotprimitivecor}.   That any two such algebras are nonisomorphic follows from Lemma \ref{cardinalitylemma}.  (Indeed, we note that in this case these algebras are in fact not even isomorphic {\it as rings}, by a similar cardinality argument.)
\end{proof}

In a completely analogous manner, we achieve as well the second aforementioned class of prime, non-primitive, von Neumann regular algebras.

\begin{thm}\label{primenotprimitivevNrthmusingEsubkappa}
Let $K$ be any field.  Let $\kappa$ be an ordinal with uncountable cofinality, and let $E_{\kappa}$ be the graph presented in Definition \ref{Esubkappadefinition}.
 Then $L_K(E_\kappa)$ is a prime, non-primitive, von Neumann regular ring.

 In particular, let $\{\kappa_\alpha \ | \ \alpha \in \mathcal{A}\}$ denote a collection of  uncountable ordinals, each of uncountable cofinality, and each having different cardinality.   Then the collection $\{L_\mathbb{C}(E_{\kappa_\alpha}) \ | \ \alpha \in \mathcal{A}\}$ consists of pairwise nonisomorphic, prime, non-primitive, von Neumann regular $\mathbb{C}$-algebras.
\end{thm}

\begin{rem}\label{remarkthatclassesarenotthesameuptoisomorphism}
{\rm We note that the  algebras described in Theorems \ref{primenotprimitivevNrthmusingEsubTsubX} and  \ref{primenotprimitivevNrthmusingEsubkappa} are different up to isomorphism one from the other; that is, $L_K(E_{\mathcal{F}(X)}) \not\cong L_K(E_\kappa)$ for any uncountable set $X$ and any uncountable ordinal $\kappa$.  The reason is as follows (see e.g. \cite{T} for further details):  since each of the graphs $E_{\mathcal{F}(X)}$ and $E_\kappa$ has Condition (K), every ideal in each of the algebras $L_K(E_{\mathcal{F}(X)})$ and $L_K(E_\kappa)$ is graded.  By \cite[Theorem 5.7]{T}, for any graph $E$ the lattice of graded ideals of $L_K(E)$ is isomorphic to the lattice $\mathcal{L}_E$ of pairs $(H,S)$, where $H$ is a hereditary saturated subset of $E^0$ and $S$ is a subset of the breaking vertex set $B_H$.   It is routine to show that for $E= E_{\mathcal{F}(X)}$ there are ${\rm card}(X)$ (proper) maximal elements in $\mathcal{L}_E$, while for $F = E_\kappa$ there is exactly one (proper) maximal element in $\mathcal{L}_F$. Thus the ideal lattices of algebras in the two classes are distinct, and therefore any two such algebras are nonisomorphic.

}

\end{rem}

 For those readers whose tastes run more towards unital rings, we now show that the previously offered examples may be slightly and easily modified to produce germane classes of such algebras.  Let $K$ be a field and let $R$ be any $K$-algebra (not necessarily with multiplicative identity).  The {\it $K$-algebra  unitization of} $R$ is the $K$-algebra
$$\widehat{R}=K\oplus R,$$ with coordinate addition and $K$-scalar action, and with ring multiplication given by setting  $$(k,r)\cdot (l,s)=(kl,ks+lr+rs)$$ for
any $(k,r),(l,s)\in \widehat{R}$.   Then $\widehat{R}$ is a $K$-algebra with identity $(1,0)$, and $R$ embeds as a two-sided algebra ideal of $\widehat{R}$ of codimension 1, via
 the map $r\mapsto (0,r)$.
Let
$$T = \{y\in \widehat{R} \ | \ y(0,I) = \{(0,0)\} \}$$
   for some nonzero two-sided ideal  $I$ of $R$.  It is shown in the proof of \cite[Lemma 2]{LRS} that if $R$ is prime, then $T$ is a prime $K$-algebra ideal of $\widehat{R}$, and that $T \cap (0,R) = \{(0,0)\}$.  (The primeness of $R$ is needed to show that $T$ is closed under addition.)   Indeed, the unital overring $R_1$ of $R$ utilized in previous sections of the current article is precisely the quotient ring $\widehat{R}/T$.

We show now that the $K$-algebra unitizations of appropriate Leavitt path algebras provide a class of examples of unital, prime, non-primitive, von Neumann regular algebras.

\begin{lem}\label{conditionspasstowidehatRlemma}
Let $K$ be any field, and $R$ any $K$-algebra.

\begin{enumerate}
\item  $R$ is prime and nonunital  if and only if $\widehat{R}$ is prime.
 \item $R$ is primitive and nonunital if and only if $\widehat{R}$ is primitive.
     \item $R$ is von Neumann regular if and only if $\widehat{R}$ is von Neumann regular.
 \end{enumerate}
\end{lem}

\begin{proof}
Statement (1) follows from \cite[Theorem 28]{Mes} (or from the more focused statement \cite[Corollary 13]{Mes2}), while   Statement (2) can be deduced from Lemma \ref{LRSLemmas} or \cite[Corollary 17]{Mes2}.  For Statement (3), it is easy to show that any ideal of a von Neumann regular ring is itself von Neumann regular.  Conversely, since $\widehat{R}/R \cong K$ is von Neumann regular, then the regularity of $R$ yields the regularity if $\widehat{R}$ by a standard argument.
\end{proof}

We now produce our third class of prime, non-primitive, von Neumann regular algebras.
\begin{thm}\label{unitalexamplesEsubTsubX}
Let $K$ be any field, and let $X$ be an uncountable set.  Then $\widehat{L_K(E_{\mathcal{F}(X)})}$ is a unital, prime, non-primitive, von Neumann regular ring.

 In particular, let $\{X_\alpha \  | \ \alpha \in \mathcal{A}\}$ denote a collection of uncountable sets, each of different cardinality.   Then the collection $\{\widehat{L_\mathbb{C}(E_{\mathcal{F}({X_\alpha)}})} \  | \ \alpha \in \mathcal{A}\}$ consists of pairwise nonisomorphic,  unital, prime, non-primitive, von Neumann regular $\mathbb{C}$-algebras.
\end{thm}

\begin{proof}
Since each $E_{\mathcal{F}(X_\alpha)}$ has infinitely many vertices and satisfies Condition (MT3),  we have that
$\widehat{L_K(E)}$ is prime by Theorem \ref{primenesstheorem} and Lemma \ref{conditionspasstowidehatRlemma}(1).  Since $E_{\mathcal{F}(X_\alpha)}$ is acyclic we get $L_K(E_{\mathcal{F}(X_\alpha)})$ is von Neumman regular by \cite[Theorem 1]{AR}, and thus $\widehat{L_K(E_{\mathcal{F}(X_\alpha)})}$ is von Neumann regular by Lemma \ref{conditionspasstowidehatRlemma}(3).  Finally, Lemma \ref{conditionspasstowidehatRlemma}(2) along with Proposition \ref{notCSPimpliesnotprimitiveprop} yields that $\widehat{L_K(E_{\mathcal{F}(X_\alpha)})}$ is not primitive.

That the algebras are pairwise nonisomorphic follows from Lemma \ref{cardinalitylemma} and the fact that $\widehat{L_K(E_{\mathcal{F}(X_\alpha)})}$ is a one-dimensional extension of $L_K(E_{\mathcal{F}(X_\alpha)})$.
\end{proof}

Arguing in exact analogy to  the proof of Theorem \ref{unitalexamplesEsubTsubX}, we achieve our final class of examples.

\begin{thm}\label{unitalexamplesEsubkappa}
Let $K$ be any field, and $\kappa$ any ordinal with uncountable cofinality.
 Then $\widehat{L_K(E_\kappa)}$ is a unital, prime, non-primitive, von Neumann regular ring.

 In particular, let $\{\kappa_\alpha \  | \ \alpha \in \mathcal{A}\}$ denote a collection of  uncountable ordinals, each of uncountable cofinality, and each of different cardinality.     Then the collection $\{\widehat{L_\mathbb{C}(E_{\kappa_\alpha}}) \  | \ \alpha \in \mathcal{A}\}$ consists of pairwise nonisomorphic, unital, prime, non-primitive, von Neumann regular $\mathbb{C}$-algebras.
\end{thm}

\begin{rem}\label{remarkthatunitalclassesarenotthesameuptoisomorphism}
{\rm  As was done for the two classes of nonunital algebras produced at the beginning of this Section, we observe here that the unital algebras  described in Theorems \ref{unitalexamplesEsubTsubX} and  \ref{unitalexamplesEsubkappa} are different up to isomorphism one from the other; that is, $\widehat{L_K(E_{\mathcal{F}(X_\alpha)})} \not\cong \widehat{L_K(E_\kappa)}$ for any uncountable set $X$ and any uncountable ordinal $\kappa$.  This follows from Remark \ref{remarkthatclassesarenotthesameuptoisomorphism} together with  the fact that the maximal ideal structure of $\widehat{R}$ is completely determined by the maximal ideal structure of $R$ (see e.g. \cite[Corollary 18]{Mes2}.
}
\end{rem}

Immediately after posing the aforementioned ``{\it Is a regular prime ring necessarily primitive?}" question in \cite{K},  Kaplansky continued:  ``{\it It seems unlikely that the answer is affirmative, but a counter-example may have to be weird.}"  While {\it weirdness} is certainly in the eye of the beholder, we believe that the examples  of prime non-primitive regular algebras presented in this section indeed arise quite naturally.

\begin{rem}
{\rm
In addition to the graphs of the form $E_{\mathcal{F}(X_\alpha)}$ and $E_\kappa$ described above, there are many additional classes of graphs germane in the current context.  For instance, let $X$ be any infinite set, and let $\mathcal{T}_X$ denote the collection  of those subsets $A$ of $X$ for which ${\rm card}(A) < {\rm card}(X)$.    We  define a graph $E = E_{\mathcal{T}_X}$ as follows:
$$E_{\mathcal{T}_X}^0 \ = \ \mathcal{T}_X, \ E_{\mathcal{T}_X}^1 = \{e_{A,A'} \ | \ A,A'\in \mathcal{T}_X, \ \mbox{and} \ A\subsetneqq A'\},$$
  $s(e_{A,A'})=A$, and $r(e_{A,A'})=A'$ for each $e_{A,A'} \in E_{\mathcal{T}_X}^1.$   Then $E$ is acyclic and satisfies Condition (MT3);  and $E$ has CSP if and only if $X$ is at most countable.
 \hfill  $\Box$
 }
\end{rem}

We conclude with the following observation.  For any graph $E$ and field $K$ the {\it Cohn path algebra} $C_K(E)$ is the $K$-algebra having the same generators and relations as those given in Definition \ref{definition} for the Leavitt path algebra $L_K(E)$, {\it except for} the (CK2) relation. (The Cohn path algebras are the natural generalizations to graphs of the algebras $U_{1,n}$ defined in \cite[Section 5]{C}.)  We note that for  graphs $E$ of the form $E_{\mathcal{F}(X)}$ or $E_\kappa$,  the Leavitt path algebra $L_K(E)$ and Cohn path algebra $C_K(E)$ coincide, as each vertex $v$ in $E$ has $|s^{-1}(v)| = \infty$. Thus the examples of prime non-primitive von Neumann regular rings given in Theorems \ref{primenotprimitivevNrthmusingEsubTsubX} and \ref{primenotprimitivevNrthmusingEsubkappa} may be interpreted as arising from the Cohn path algebra construction as well.

\end{document}